\documentclass[11pt]{article}
\usepackage{amsmath}
\usepackage{amsfonts}
\usepackage{amssymb}
\newtheorem{theorem}{Theorem}
\newtheorem{corollary}{Corollary}

\newtheorem{proposition}{Proposition}

\newtheorem{remark}{Remark}

\usepackage[top=3.5cm, bottom=3.5cm, left=2cm,right=2cm]{geometry}

\font\QEDlogofont=msam10 at 10pt

\def\QEDblogo{\hbox{\QEDlogofont\char'004}}
\newif\ifnologo\nologofalse
\newif\iflogo
\newif\ifblogo\blogofalse
\newif\iftopprhead\topprheadfalse

\def\prooffont{\normalsize}
\newenvironment{proof}{\par\addvspace{6pt plus2pt}
\par
\noindent\prooffont{\bf\em Proof:}\hskip6pt\ignorespaces}{%
   \ifblogo\hskip1.2pt
            \blacksquare
   \else
   \ifnologo
   \else
   \hfill
            \QEDblogo
   \fi\fi
\par\addvspace{6pt plus2pt}\global\topprheadfalse}%

\usepackage{amsbsy,amsfonts,amsmath,amssymb,enumerate,epsfig,graphicx,rotating,subfigure}

 \usepackage{cite}

\begin{document}

\begin{center}

\begin{large}
{\bf A new characterization of ultraspherical, Hermite, and Chebyshev polynomials of the first kind}
\end{large}
\vspace{10pt}

{\bf Mohammed Mesk $^{\rm a,}$\footnote{Email:
m\_mesk@yahoo.fr} and Mohammed Brahim Zahaf $^{\rm
b,c,}$\footnote{Email:
m\_b\_zahaf@yahoo.fr}}\vspace{6pt}
\\ \vspace{6pt} $^{\rm a}${\em Laboratoire d'Analyse Non Lin\'eaire et Math\'ematiques Appliqu\'ees,\break
Universit\'e de Tlemcen, BP 119,  13000-Tlemcen, Alg\'erie.
}\\

\vspace{6pt} $^{\rm b}${\em D\'epartement de Math\'ematiques,
Facult\'e des sciences, \break
Universit\'e de Tlemcen, BP 119,  13000-Tlemcen, Alg\'erie.
}\\

\vspace{6pt} $^{\rm c}${\em 
Laboratoire de Physique Quantique de la Mati\`ere \break
et Mod\'elisations Math\'ematiques (LPQ3M)\break
 Universit\'e de Mascara,
29000-Mascara, Alg\'erie.}
\end{center}
\vspace{1cm}
\begin{abstract}
We show that the only polynomial sets with a generating function of the form $F(xt-R(t))$ and satisfying a three-term recursion relation are the monomial set and the rescaled ultraspherical, Hermite, and Chebyshev polynomials of the first kind.

\vspace{6pt}


\vspace{6pt}

\end{abstract}

\section{Introduction and main result}

The problem of describing all or just orthogonal polynomials generated by a specific generating function has been investigated by many authors (see for example \cite{Appel, alsalam,alsalam2,anshelev,Bachhaus,Boas, ch1,ch2,Meixner}). For the
special case, where the generating function has the form $F(xt-\alpha t^2)$, the authors in \cite{alsalam}, \cite{Bachhaus} and \cite{Bencheikh} used different methods to show that the orthogonal polynomials are Hermite and ultraspherical polynomials. Recently in \cite{anshelev}, the author gave a motivation of this question and found, even if $F$ is a formal power series, that the orthogonal polynomials are the ultraspherical, Hermite and Chebychev polynomials of the first kind. Moreover, for $F$ corresponding to Chebychev polynomials of the first kind, he showed that these polynomials remain the only orthogonal polynomials with generating function of the form $F(xU(t)-R(t))$, where $U(t)$ and $R(t)$ are formal power series. A natural question, as mentioned in \cite{anshelev}, is to describe (all or just orthogonal) polynomials with generating functions $$F(xU(t)-R(t)).$$

In this paper, we consider the subclass case
$F(xt-R(t))=\sum_{n\geq 0}\alpha_nP_n(x)t^n$ where the polynomial set (not necessary orthogonal) $\{P_n\}_{n\geq 0}$
satisfies a three-term recursion relation. The main result
obtained here is the following:
\begin{theorem}\label{thmm}
Let $F(t)=\sum_{n\geq 0}\alpha_nt^n$ and $R(t)=\sum_{n\geq
1}R_nt^n/n$ be formal power series where $\{\alpha_n\}$ and
$\{R_n\}$ are complex sequences with $\alpha_0=1$ and $R_1=0$.
Define the polynomial set $\{P_n\}_{n\geq 0}$ by
\begin{equation}\label{gf0}
F(xt-R(t))=\sum_{n\geq 0}\alpha_nP_n(x)t^n.
\end{equation}
If this polynomial set (which is automatically monic) satisfies the three-term recursion
relation
\begin{eqnarray}\label{gf6}
\left\{
\begin{array}{l}
xP_n(x)=P_{n+1}(x)+\beta_nP_n(x)+\omega_nP_{n-1}(x),\quad n\geq
0,\\
P_{-1}(x)=0,\;\;P_0(x)=1
\end{array}
\right.
\end{eqnarray} 
where $\{\beta_n\}$ and $\{\omega_n\}$ are complex sequences,
then we have:\\
a) If $R_2=0$ and $\alpha_n\neq 0$ for $n\geq 1$, then $R(t)=0$, $F(t)$ is arbitrary and $F(xt)=\sum_{n\geq 0}\alpha_nx^nt^n$
generates the monomials $\{x^n\}_{n\geq 0}$.\\
b) If $\alpha_1R_2\neq 0,$ then $R(t)=R_2t^2/2$ and the polynomial sets $\{P_n\}_{n\geq 0}$ are the rescaled ultraspherical, Hermite and Chebychev polynomials of the first kind.
 \end{theorem}

In the above theorem, let remark that there is no loss of generality in assuming $\alpha_0=1$ and $R_1=0$. Indeed, we can choose the generating function $\gamma_1+\gamma_2F((x+R_1)t-R(t))=\gamma_1+\gamma_2\sum_{n\geq 0}\alpha_nP_n(x+R_1)t^n$ for suitable constants  $\gamma_1$ and $\gamma_2$.

The proof of theorem \ref{thmm} will be given in section \ref{sec3}. For that purpose some preliminary results must be developed  first in section \ref{sec2}. We end the paper by a brief concluding section.

\section{Preliminary results}\label{sec2}

This section contains two propositions and some related corollaries which are the important ingredients used for the proof of Theorem~\ref{thmm}.
\begin{proposition}\label{prop1}
Let $\{P_n\}_{n\geq 0}$ be a monic polynomial set generated by
\eqref{gf0}. Then we have

\begin{equation}\label{gf7}
\alpha_nxP'_n(x)-\sum_{k=1}^{n}R_{k+1}\alpha_{n-k}P'_{n-k}(x)=n\alpha_nP_n(x),\;\;
n\geq 1.
\end{equation}

\end{proposition}

\begin{proof}
By combining the two derivatives $\frac{\partial W}{\partial x}$ and $\frac{\partial W}{\partial t}$ of   the generating function $W(x,t)=F(xt-R(t))$,
 we obtain
\begin{equation}\label{gf5}
\left(x-R'(t)\right)\frac{\partial W}{\partial
x}=t\frac{\partial W}{\partial t}.
\end{equation}
The substitution of the right-hand side of \eqref{gf0} and
$R'(t)=\sum_{n\geq 0}\ R_{n+1}t^{n}$ in \eqref{gf5} gives
\begin{equation}\label{gf5.1}
\left(x-\sum_{n\geq 0}\ R_{n+1}t^n\right)\sum_{n\geq
0}\alpha_nP'_n(x)t^n=\sum_{n\geq 0}\alpha_nP_n(x)nt^{n}.
\end{equation}
After a resummation procedure in left hand side, namely: $$\left(\sum_{n\geq 0}\
R_{n+1}t^{n}\right)\left(\sum_{n\geq
0}\alpha_nP'_n(x)t^n\right)=\sum_{n\geq
0}\left(\sum_{k=0}^{n}R_{k+1}\alpha_{n-k}P'_{n-k}(x)\right)t^n$$ and a $t^n$  coefficients comparison in \eqref{gf5.1}, the result
\eqref{gf7} of proposition 1 follows.

\end{proof}

\begin{corollary}\label{cor0} 
Let $\{P_n\}_{n\geq 0}$ be a monic polynomial set generated by
\eqref{gf0}. If $\alpha_1R_2\neq 0$ then $\alpha_n\neq 0$ for
$n\geq 2$.
\begin{proof}
In fact suppose that $\alpha_{n_0}=0$ for an $n_0\geq 2$. Then \eqref{gf7} implies that $R_{k+1}\alpha_{n_0-k}=0$ for $k=1,...,n_0-1.$
In particular $R_{2}\alpha_{n_0-1}=0$ for $k=1$ gives $\alpha_{n_0-1}=0$ since $R_2\neq 0$. By induction we arrive at $\alpha_1=0$ which 
contradicts the premise $\alpha_1\neq 0$.
\end{proof}
\end{corollary}

\begin{corollary}\label{cor1}
Let $\alpha_1R_2\neq 0$. If the polynomial set $\{P_n\}$ generated by
\eqref{gf0} is symmetric, then
$$R_{2l+1}=0,\;\;\; \text{ for } l\geq 1.$$
\end{corollary}
\begin{proof}
The polynomial set $\{P_n\}$ is symmetric means that $P_n(-x)=(-1)^nP_n(x)$ for $n\geq 0$. The substitution $x\rightarrow -x$ in equation
\eqref{gf7} minus equation \eqref{gf7} itself 
left us with $$(1-(-1)^{k+1})R_{k+1}\alpha_{n-k}=0, \;\; \hbox { for } 1\leq k\leq n-1.$$

So, by Corollary \ref{cor0}, we have $R_{2l+1}=0$, for $ l\geq 1$, and
$R(t)=\sum_{k\geq 1}\frac{R_{2k}}{2k}t^{2k}$.
\end{proof}

\begin{proposition}\label{prop2}
Let $\alpha_1R_2\neq 0$ and define 
\begin{equation}\label{define}
T_k=R_{2k},\;(k\geq 1),\;a_n=\frac{T_1}{2}\frac{\alpha_n}{\alpha_{n+1}}, \;(n\geq
0)\text{ and } c_n=\frac{\alpha_n}{\alpha_{n-1}}\,\omega_n,\;\;(n\geq 1).
\end{equation}
For the monic polynomial set generated by \eqref{gf0} and satisfying \eqref{gf6} we have:\\
a) 

\begin{equation}\label{gf9}
\beta_n=0,\;\;\text{  for } n\geq 0.
\end{equation}
b)
\begin{equation}\label{gf10}
\omega _n=na_n-(n-1)a_{n-1},\;\;\text{ for } n\geq 1.
\end{equation}
c)
\begin{equation}\label{gf11}
\frac{4T_2}{T_1^3}\left(1-\frac{n-3}{n-2}\frac{a_{n-3}}{a_{n}}\right)=\frac{n+1}{a_n}-\frac{2n}{a_{n-1}}+\frac{n-1}{a_{n-2}},\;\;\text{
for } n\geq 3.
\end{equation}
d)

\begin{eqnarray}\label{gf12}
\frac{2}{T_1}\left(a_n-\frac{n-2k-1}{n-2k}a_{n-2k-1}\right)T_{k+1}+\left(\frac{n+2}{n}c_n-\frac{n-2k+1}{n-2k+2}c_{n-2k+1}\right)T_{k}=\nonumber\\
=\sum_{l=1}^{k}\frac{T_{l}T_{k-l+1}}{n-2k+2l},\;\; \text{ for }
k\geq 2 \text{ and } n\geq 2k+1.
\end{eqnarray}
\end{proposition}

\begin{proof}
By differentiating \eqref{gf6} we get
\begin{equation}\label{gf13}
xP'_n(x)+P_n(x)=P'_{n+1}(x)+\beta_nP'_n(x)+\omega_nP'_{n-1}.
\end{equation}
Then by making the combinations $n\alpha_nEq\eqref{gf13}+Eq\eqref{gf7}$ and $Eq\eqref{gf7}-\alpha_n
Eq\eqref{gf13}$ we obtain, respectively,

\begin{eqnarray}\label{gf14}
(n+1)\alpha_nxP'_n(x)=n\alpha_n\left(P'_{n+1}(x)+\beta_nP'_n(x)+\omega_nP'_{n-1}(x)\right)+\sum_{k=1}^{n-1}R_{k+1}\alpha_{n-k}P'
_{n-k}(x)
\end{eqnarray}
and
\begin{eqnarray}\label{gf15}
(n+1)\alpha_n
P_n(x)=\alpha_n\left(P'_{n+1}(x)+\beta_nP'_n(x)+\omega_nP'_{n-1}(x)\right)-\sum_{k=1}^{n-1}R_{k+1}\alpha_{n-k}P'
_{n-k}(x).
\end{eqnarray}

Multiplying \eqref{gf15} by $x$ and using \eqref{gf6} in the left-hand side gives
\begin{equation}\label{xgf15}
(n+1)\alpha_n\left(P_{n+1}(x)+\beta_nP_n(x)+\omega_nP_{n-1}(x)\right)=\alpha_n\left(xP_{n+1}'(x)+\beta_n xP_n'(x)+\omega_nxP_{n-1}'(x)\right)-\sum_{k=1}^{n-1}R_{k+1}\alpha_{n-k}xP_{n-k}'(x).
\end{equation}

For the left-hand side (resp. the right-hand side) of \eqref{xgf15} we use \eqref{gf15} (resp. \eqref{gf14}) to get
\begin{eqnarray}\label{xgf1510}
&&\frac{n+1}{n+2}\alpha_n\left(P_{n+2}'(x)+\beta_{n+1}P_{n+1}'(x)+\omega_{n+1}P_{n}'(x)\right)-\frac{n+1}{n+2}\frac{\alpha_n}{\alpha_{n+1}}\sum_{k=1}^{n}R_{k+1}\alpha_{n-k+1}P_{n-k+1}'(x)\nonumber\\
&&+\alpha_n\beta_n\left(P_{n+1}'(x)+\beta_nP_{n}'(x)+\omega_nP_{n-1}'(x)\right)-\beta_n\sum_{k=1}^{n-1}R_{k+1}\alpha_{n-k}P_{n-k}'(x)\nonumber\\
&&+\frac{n+1}{n}\alpha_n\omega_n\left(P_{n}'(x)+\beta_{n-1}P_{n-1}'(x)+\omega_{n-1}P_{n-2}'(x)\right)-\frac{n+1}{n}\frac{\alpha_n}{\alpha_{n-1}}\omega_n\sum_{k=1}^{n-2}R_{k+1}\alpha_{n-k-1}P_{n-k-1}'(x)=\nonumber\\
&&=\frac{n+1}{n+2}\alpha_n\left(P_{n+2}'(x)+\beta_{n+1}P_{n+1}'(x)+\omega_{n+1}P_{n}'(x)\right)+\frac{1}{n+2}\frac{\alpha_n}{\alpha_{n+1}}\sum_{k=1}^{n}R_{k+1}\alpha_{n-k+1}P_{n-k+1}'(x)
\nonumber\\
&&+\frac{n}{n+1}\alpha_n\beta_n\left(P_{n+1}'+\beta_nP_{n}'+\omega_nP_{n-1}'\right)
+\frac{1}{n+1}\beta_n\sum_{k=1}^{n-1}R_{k+1}\alpha_{n-k}P_{n-k}'(x)
\nonumber\\
&&+\frac{n-1}{n}\alpha_n\omega_n\left(P_{n}'(x)+\beta_{n-1}P_{n-1}'(x)+\omega_{n-1}P_{n-2}'(x)\right)
+\frac{1}{n}\frac{\alpha_n}{\alpha_{n-1}}\omega_n\sum_{k=1}^{n-2}R_{k+1}\alpha_{n-k-1}P_{n-k-1}'(x)\nonumber\\
&&-\sum_{k=1}^{n-1}R_{k+1}\frac{n-k}{n-k+1}\alpha_{n-k}\left(P_{n-k+1}'(x)+\beta_{n-k}P_{n-k}'(x)+\omega_{n-k}P_{n-k-1}'(x)\right)
\nonumber\\
&&-\sum_{k=1}^{n-1}\frac{R_{k+1}}{n-k+1}\sum_{l=1}^{n-k-1}R_{l+1}\alpha_{n-k-l}P_{n-k-l}'(x),
\end{eqnarray}
which can be simplified to
\begin{eqnarray}\label{xgf151}
&&-\frac{\alpha_n}{\alpha_{n+1}}\sum_{k=1}^{n}R_{k+1}\alpha_{n-k+1}P_{n-k+1}'(x)+\frac{1}{n+1}\alpha_n\beta_n\left(P_{n+1}'(x)+\beta_nP_{n}'(x)+\omega_nP_{n-1}'(x)\right)\nonumber\\
&&-\frac{n+2}{n+1}\beta_n\sum_{k=1}^{n-1}R_{k+1}\alpha_{n-k}P_{n-k}'(x)
+\frac{2}{n}\alpha_n\omega_n\left(P_{n}'(x)+\beta_{n-1}P_{n-1}'(x)+\omega_{n-1}P_{n-2}'(x)\right)\nonumber\\
&&-\frac{n+2}{n}\frac{\alpha_n}{\alpha_{n-1}}\omega_n\sum_{k=1}^{n-2}R_{k+1}\alpha_{n-k-1}P_{n-k-1}'(x)
+\sum_{k=1}^{n-1}R_{k+1}\frac{n-k}{n-k+1}\alpha_{n-k}(P_{n-k+1}'(x)+\beta_{n-k}P_{n-k}'(x)\nonumber\\
&&+\omega_{n-k}P_{n-k-1}'(x))
+\sum_{k=1}^{n-1}\frac{R_{k+1}}{n-k+1}\sum_{l=1}^{n-k-1}R_{l+1}\alpha_{n-k-l}P_{n-k-l}'(x)=0.
\end{eqnarray}

From \eqref{xgf151}, the coefficient of $P'_{n+1}(x)$ is null, so we get \eqref{gf9} which  means that the polynomial set $\{P_n\}$ is symmetric, see \cite[Theorem 4.3]{ch3}. Therefore, by Corollary \ref{cor1}, the odd part of the $R$-sequence is null and a computation of the coefficients of $P'_{n}(x)$, $P'_{n-2}(x)$ and $\{P'_{n+1-k}(x)\}_{n\geq k\geq 4}$ in \eqref{xgf151} yields
\begin{equation}\label{xgfOmega}
\frac{2}{n}\alpha_n\omega_n=R_2\frac{\alpha_n}{\alpha_{n+1}}\alpha_n-R_2\frac{n-1}{n}\alpha_{n-1},\quad \text{for }n\geq 1,
\end{equation}
\begin{equation}\label{xgf153}
\frac{2}{n}\alpha_n\omega_n\omega_{n-1}=R_{4}\frac{\alpha_n}{\alpha_{n+1}}\alpha_{n-2}+R_{2}\frac{n+2}{n}\frac{\alpha_n}{\alpha_{n-1}}\alpha_{n-2}\omega_n-R_{4}\frac{n-3}{n-2}\alpha_{n-3}-R_{2}\frac{n-1}{n}\alpha_{n-1}\omega_{n-1}
-\frac{R_{2}^2}{n}\alpha_{n-2},\quad \text{for }n\geq 3,
\end{equation}
and
\begin{eqnarray}\label{xgf156}
R_{k+1}\left(\frac{\alpha_n}{\alpha_{n+1}}-\frac{n-k}{n-k+1}\frac{\alpha_{n-k}}{\alpha_{n-k+1}}\right)+R_{k-1}\left(\frac{n+2}{n}\frac{\alpha_n}{\alpha_{n-1}}\,\omega_n
-\frac{n-k+2}{n-k+3}\frac{\alpha_{n-k+2}}{\alpha_{n-k+1}}\omega_{n-k+2}\right)\nonumber\\
=\sum_{l=1}^{k-2}\frac{R_{k-l}R_{l+1}}{n-k+l+2},\;\;n\geq k\geq 5.
\end{eqnarray}
respectively.\\
Finally, by using the notations \eqref{define}, substituting for $\omega_n$ from \eqref{xgfOmega} into \eqref{xgf153} and by shifting $(k,l)\rightarrow(2k+1,2l-1)$ in \eqref{xgf156} we obtain \eqref{gf10}, \eqref{gf11} and \eqref{gf12}.

\end{proof}

In the following corollaries we adopt the same conditions and notations of Proposition \ref{prop2}.

\begin{corollary}\label{cor2}
If $T_2=0$ then $R(t)=T_1t^2/2$. In this case, the polynomials
generated by $F(xt-T_1t^2/2)$ and satisfying (\ref{gf6}) reduce to the rescaled ultraspherical, Hermite and Chebychev polynomials
of the first kind.
  \end{corollary}

\begin{proof} We will use \eqref{gf12} and proceed by induction on
$k$ to show that $T_k=0$ for $k\geq 3$. Indeed $k=2$ and $n=5$
in \eqref{gf12} leads to $2a_5T_3/T_1=0$ and since $a_n\neq 0$ by Corollary \ref{cor0} we get $T_3=0$.
Suppose that $T_3=T_4=\ldots=T_k=0$. Then for $n=2k+1$ the
equation \eqref{gf12} gives $2a_{2k+1}T_{k+1}/T_1=0$ and finally $T_{k+1}=0$. Accordingly, $R(t)=T_1t^2/2$ and the generating function \eqref{gf0} takes the form $F\left(xt-T_1t^2/2\right)$.

Actually to determine $F(t)$, we make use of \eqref{gf11} with $T_1\neq 0$ and $T_2=0$ :
\begin{equation}\label{ricati}
\frac{n+1}{a_n}-\frac{2n}{a_{n-1}}+\frac{n-1}{a_{n-2}}=0, \quad
n\geq 3.
\end{equation}
By summing twice in \eqref{ricati} we find
\begin{equation}\label{solricati}
\frac{n+1}{a_n}=\left(\frac{3}{a_2}-\frac{2}{a_1}\right)n+\left(\frac{4}{a_1}-\frac{3}{a_2}\right),
\quad n\geq 3,
\end{equation}
which is valid for $n=1$ and $n=2$. As $a_n=\frac{T_1}{2}\frac{\alpha_n}{\alpha_{n+1}}$ and using \eqref{solricati} we get, for  $n\geq 2$,
\begin{equation}\label{alphan}
\alpha_{n}=\frac{\lambda_2(n-1)+\lambda_1}{n}\alpha_{n-1}=\frac{\prod_{j=1}^{n-1}(\lambda_1+j\lambda_2)}{n!}\alpha_1.
\end{equation}
And the generating function reads
$$F(t)=1+\alpha_{1}t+\alpha_1\sum_{n\geq
2}\frac{\prod_{j=1}^{n-1}(\lambda_1+j\lambda_2)}{n!}t^n,$$
where $\lambda_1=4\alpha_2/\alpha_1-3\alpha_3/\alpha_2$ and
$\lambda_2=3\alpha_3/\alpha_2-2\alpha_2/\alpha_1$.

Now, by the same ideas as in \cite{anshelev} we have:

i) If $\lambda_1\neq 0$, $\lambda_2\neq 0$, then
\begin{equation}\label{sol1}
F(t)=1+\frac{\alpha_{1}}{\lambda_1}\left(\frac{1}{(1-\lambda_2 t)^{\lambda_1/\lambda_2}}-1\right).
\end{equation}
According to \eqref{alphan}, the ratio
$\lambda:=\lambda_1/\lambda_2$ can not be a negative integer and consequently the
monic polynomials ${P_n}$ can be written as
\begin{equation}\label{sol1ultra}
P_n(x)=\left(\frac{2T_1}{\lambda_2}\right)^{n/2}\textbf{C}_{n}^{\lambda}\left(\sqrt{\frac{\lambda_2}{2T_1}}x\right),
\end{equation}
where $\textbf{C}_{n}^{\lambda}(x)$ are the monic ultraspherical
polynomials defined by \cite[Section 9.8.1]{koekoek}
$$\frac{1}{(1-2xt+t^2)^{\lambda}}=\sum_{n\geq
0}\frac{(-2)^n(\lambda)_{n}}{n!}\textbf{C}_{n}^{\lambda}(x)t^n.$$

ii) If $\lambda_1=0$, $\lambda_2=2\alpha_2/\alpha_1$, then
\begin{equation}\label{sol2}
F(t)=1+\frac{\alpha_{1}}{\lambda_2}\ln\left(\frac{1}{1-\lambda_2 t}\right).
\end{equation}
This later function generates
\begin{equation}\label{sol2Chebichev}
P_n(x)=\left(\frac{2T_1}{\lambda_2}\right)^{n/2}\textbf{T}_{n}\left(\sqrt{\frac{\lambda_2}{2T_1}}x\right),
\end{equation}
where $\textbf{T}_n$ are the monic Chebyshev polynomials defined by  \cite[p. 155]{ch3}
$$1+\frac{1}{2}\ln\left(\frac{1}{1-2xt+t^2}\right)=1+\sum_{n=
1}^{\infty}\frac{2^{n-1}}{n}\textbf{T}_{n}(x)t^n.$$

iii) If $\lambda_1=2\alpha_2/\alpha_1$,
$\lambda_2=0$, the generating function reads
\begin{equation}\label{sol3}
F(t)=1+\frac{\alpha_{1}}{\lambda_1}\left(e^{\lambda_1t}-1\right)
\end{equation}
and the polynomials $P_n$ take the form
\begin{equation}\label{sol3Hermite}
P_n(x)=\left(\frac{T_1}{\lambda_1}\right)^{n/2}\textbf{H}_{n}\left(\sqrt{\frac{\lambda_1}{T_1}}x\right),
\end{equation}
where $\textbf{H}_n$ are the monic Hermite polynomials defined by \cite[Section 9.15]{koekoek}
$$\exp(xt-t^2/2)=\sum_{n= 0}^{\infty}\textbf{H}_{n}(x)\frac{t^n}{n!}.$$

\end{proof}

\begin{remark}\label{remark3}
Suppose $\lambda_1$ and $\lambda_2$ are real numbers then the orthogonality of these polynomials requires  $w_n>0$ for all $n\geq 1$, where
\begin{equation}
\omega_n=\frac{{ T_1}}{2}\,{\frac {n\, \left( {\lambda_2}\,(n-1) +2\,{\lambda_1}\right) }{ \left( {\lambda_2}\,n+{\lambda_1} \right) 
 \left( {\lambda_2}\,(n-1)+{\lambda_1} \right) }}.
\end{equation}

Then it is enough to  assume $\lambda_2/T_1>0$ and $\lambda>-1/2$, $\lambda_2/T_1>0$ and $\lambda_1/T_1>0$ for the ultraspherical, Chebyshev and Hermite polynomials cases, respectively.
\end{remark}

\begin{corollary}\label{cor3}
If $T_{\kappa}=T_{\kappa+1}=0$ for some $\kappa\geq 3$, then
$T_2=0$.
\end{corollary}

\begin{proof}

Let $k=\kappa$ in \eqref{gf12}. Then for $n\geq 2\kappa +1$ the
fraction $\sum_{l=1}^{\kappa}\frac{T_lT_{\kappa-l+1}}{n-2\kappa+2l}$, as function of integer $n$, is
null even for real $n$. Multiplying by $n-2\kappa+2l$ and tends $n$ to $2\kappa-2l$ we find
$T_lT_{\kappa-l+1}=0$ for $1\leq l \leq \kappa$ which is $T_2T_{\kappa-1}=0$ when $l=2$. Supposing  $T_2\neq 0$ leads to
$T_{\kappa-1}=0$. So $T_{\kappa-1}=T_{\kappa}=0$ and with the
same procedure we find $T_{\kappa-2}=0$. Going so on till we arrive at $T_2=0$ which contradicts $T_2\neq 0$.
\end{proof}
\begin{corollary}\label{cor4}
 If $R(t)$ is a polynomial then $R(t)=T_1t^2/2$.
 \end{corollary}

\begin{proof}
If $R(t)$ is a polynomial then for some $\kappa\geq 2$, $T_k=0$
whenever $k\geq \kappa$. By
Corollary~\ref{cor3}, since $T_{\kappa}=T_{\kappa+1}=0$, we conclude that $T_2=0$ and  by
Corollary~\ref{cor2} that $T_k=0$ for $k\geq 3$.
\end{proof}

\begin{corollary}\label{cor5}
 If $a_n$ is a rational fraction of $n$ then $T_2=0$.
\end{corollary}

\begin{proof} Observe that $c_n=T_1\left(na_n/a_{n-1}-(n-1)\right)/2$ will also a rational fraction of $n$. Then it follows that, in \eqref{gf12}, two fractions are equal for natural numbers $n\geq 2k+1$, $k\geq 2$ and extensively will be for real numbers $n$. If we denote by $N_s(F(x))$ the number of singularities of a rational fraction
$F(x)$ then we can easily verify, for all rational fractions  $F$ and $\tilde{ F}$ of $x$ and a constant $a\neq 0$, that:

a) $N_s(F(x+a))=N_s(F(x))$,

b) $N_s(aF(x))=N_s(F(x))$,

c) $N_s(F(x)+\tilde{ F}(x))\leq N_s(F(x))+N_s(\tilde{ F}(x))$.

Using  property a) of $N_s$ we have $$N_s\left(\frac{n-2k-1}{n-2k}a_{n-2k-1}\right)=N_s\left(\frac{n}{n+1}a_{n}\right) \text{ and } N_s\left(\frac{n-2k+1}{n-2k+2}c_{n-2k+1}\right)=N_s\left(\frac{n}{n+1}c_{n}\right).$$ 
According to properties b) and c) of $N_s$, the $N_s$ of the left-hand side of \eqref{gf12} is finite and independent of $k$. Thus, the right-hand side of \eqref{gf12} has a finite number  of
singularities which is independent of $k$. As consequence there exists a $k_0$ for which $T_lT_{k-l+1}=0$
 for all $k \geq k_0$  and $k_0\leq l\leq k$. Taking successively  $k=k_0=l$ and  $k=k_0+1=l$ we get
$T_{k_0}=T_{k_0+1}=0$. Then, by Corollary~\ref{cor3} we have $T_2=0$.

\end{proof}

\begin{remark}
The fact that $a_n$ is a rational fraction of $n$ means that $F(z)=\sum_{n\geq
0}\alpha_nz^n$ is a series of hypergeometric type.
\end{remark}
\begin{corollary}\label{cor6}
If $T_{\kappa}=T_m=0$ for some $\kappa\neq m\geq 3$, then
$T_2=0$.
\end{corollary}
\begin{proof}
If $T_{\kappa +1}=0$ or $T_{m+1}=0$ then by Corollary \ref{cor3} we have $T_2=0$. Suppose that $T_{\kappa+1}\neq 0$ and
$T_{m+1}\neq 0$. Take $k=\kappa$ and $k=m$ in \eqref{gf12} to
get, respectively,
\begin{equation}\label{gf16}
\frac{2}{T_1}\left(a_n-\frac{n-2\kappa-1}{n-2\kappa}a_{n-2\kappa-1}\right)T_{\kappa+1}=\sum_{l=1}^{\kappa}\frac{T_lT_{\kappa-l+1}}{n-2\kappa+2l},
\text{ for } n\geq 2\kappa+1,
\end{equation}
and 
\begin{equation}\label{gf17}
\frac{2}{T_1}\left(a_n-\frac{n-2m-1}{n-2m}a_{n-2m-1}\right)T_{m+1}=\sum_{l=1}^{m}\frac{T_lT_{m-l+1}}{n-2m+2l},
\text{ for } n\geq 2m+1.
\end{equation}
The operation
$Eq\eqref{gf16}/T_{\kappa+1}-Eq\eqref{gf17}/T_m$
gives
\begin{equation}\label{gf18}
\frac{n-2m-1}{n-2m}a_{n-2m-1}-\frac{n-2\kappa-1}{n-2\kappa}a_{n-2\kappa-1}=Q_1(n).
\end{equation}
Assuming $m>\kappa$ and replacing $n$ by $n+2m+1$ (resp. $n+2m-2\kappa$) in \eqref{gf18} (resp. \eqref{gf17}) leads to
\begin{equation}\label{gf19}
\frac{n}{n+1}a_{n}-\frac{n+2m-2\kappa}{n+2m-2\kappa+1}a_{n+2m-2\kappa+1}=Q_1(n+2m+1)
\end{equation}
and
\begin{equation}\label{gf20}
a_{n+2m-2\kappa}-\frac{n-2\kappa-1}{n-2\kappa}a_{n-2\kappa-1}=Q_2(n+2m-2\kappa).
\end{equation}
Now $Eq\eqref{gf16}/T_{\kappa+1}- Eq\eqref{gf20}$ is
the equation
\begin{equation}\label{gf21} 
a_n-a_{n+2m-2\kappa}=Q_3(n).
\end{equation}
Multiplying \eqref{gf21} by $\frac{n+2m-2\kappa}{n+2m-2\kappa+1}$
and using \eqref{gf19} we find
\begin{equation}\label{gf22}
\left(\frac{n}{n+1}-\frac{n+2m-2\kappa}{n+2m-2\kappa+1}\right)a_n=Q_4(n).
\end{equation}
Since the $Q_i(n)$ ($i=1..4$) functions, the right-hand sides of \eqref{gf18}, \eqref{gf19}, \eqref{gf20}, \eqref{gf21} and \eqref{gf22}, are partial fractions of $n$ then  $a_n$ is also a partial fraction of $n$; and by Corollary~\ref{cor5} we deduce $T_2=0$.
\end{proof}



\begin{corollary}\label{cor7}
The following equality is true for $k\geq 3\hbox{ and } n\geq 2k+3.$
\begin{eqnarray}\label{gf23}
T_{k-1}D_{k+1}(a_n-\tilde{
a}_{n-2k-3})-T_{k+1}D_{k}(a_{n-2}-\tilde{
a}_{n-2k-1})=\sum_{l=1}^{k-1}\frac{V_{k,l}}{n-2k+2l},
\end{eqnarray}

where
\begin{itemize}
\item $D_{k,l}=T_kT_{k-l+1}-T_{k+1}T_{k-l}$.
\item $D_{k}=D_{k,1}=T_k^2-T_{k+1}T_{k-1}$.
\item
$V_{k,l}=\frac{T_1}{2}\left(T_lT_{k+1}D_{k-1,l-1}-T_{l+1}T_{k-1}D_{k,l}\right)$.
\item $\tilde{ a}_n=\frac{n}{n+1}a_n$.
\end{itemize}
\end{corollary}

\begin{proof}
Denoting the equation \eqref{gf12} by $E(k,n)$ then \eqref{gf23} is
the result of the operation
\begin{equation*}
T_{k+1}\left(T_{k-1}E(k,n)-T_{k}E(k-1,n-2)\right)-T_{k-1}\left(T_kE(k+1,n)-T_{k+1}E(k,n-2)\right).
\end{equation*}
\end{proof}

Now we are in a position to prove Theorem \ref{thmm}.

\section{Proof of Theorem \ref {thmm}}\label{sec3}

\textbf{The proof of a)} \\
As $R_1=R_2=0$, it is enough to show by induction that $R_n=0$ for $n\geq 3$.
For $n=1,2,3$, the equation \eqref{gf7} gives $P_1(x)=x$, $P_2(x)=x^2$ and $P_3(0)=-\frac{R_3\alpha_1}{3\alpha_3}$. But according to equation \eqref{gf6}, for $n=2$,  $P_3(0)=0$ and then $R_3=0$. Now assume that $R_k=0$ for $2\leq k\leq n-1$. According to \eqref{gf7} we have, for $2\leq k\leq n-1$, $P_k(0)=0$ and $P_n(0)=-\frac{R_n\alpha_1}{n\alpha_n}$.
On other hand, by the shift $n\rightarrow n-1$ in \eqref{gf6} we have $P_n(0)=0$ and thus $R_n=0$. As $R(t)=0$, the generating function \eqref{gf0} reduces to  $F(xt)=\sum_{n\geq 0}\alpha_nx^nt^n$ which generates the monomials with $F(t)$ arbitrary.

\noindent\textbf{The proof of b)}\\
According to Corollary \ref{cor2}, it is sufficient to prove that $T_2=0$. In the sequel we will investigate three cases:\\

{\bf\textit { Case 1:}} There exists $k_{0}\geq 3$ such that $D_{k}\neq 0$ for $k\geq k_{0}$.

Considering  Corollary \ref{cor6} we can choose $\tilde{k}\geq k_0$ such that $T_{k}\neq 0$ for $k\geq \tilde{k}-1$. Let, for $k\geq \tilde{k}$, $\bar{D}_{k}=\frac{D_{k}}{T_{k-1}T_{k}}$ and $\bar {E}(k,n)$ be the equation \eqref{gf23} divided by
$T_{k-1}T_{k}T_{k+1}$.

By making the operations
\begin{equation*}
\bar{D}_{k-1}\bar{E}(k,n+2)-\bar{D}_k\bar{E}(k-1,n)-\bar{D}_{k}\bar{E}(k,n)+\bar{D}_{k+1}\bar{E}(k-1,n-2),
\end{equation*}
we can eliminate $\tilde{a}_{n-2k-3}$ and $\tilde{a}_{n-2k-1}$
and keeping only, for $k\geq \tilde{k}+1$, the following equation
\begin{equation}\label{Dk101}
{a}_{n+2}-{a}_{n-4}-\tilde{D}_k({a}_n-{a}_{n-2})=\sum_{l=1}^{k}\frac{{W}_{k,l}}{n-2k+2l}:=Q_k^{(1)}(n),
\end{equation}
where $W_{k,l}$ is independent of $n$ and

$$\tilde{D}_k=\frac{\bar{D}_{k}^2+\bar{D}_{k}\bar{D}_{k-1}+\bar{D}_{k}\bar{D}_{k+1}}{\bar{D}_{k-1}\bar{D}_{k+1}}.$$

Similarly, after eliminating $a_n$ and $a_{n-2}$ by  the operations
\begin{equation}\label{Dk2.1}
\bar{D}_{k-1}\bar{E}(k,n+2)-\bar{D}_{k+1}\bar{E}(k-1,n+2)-\bar{D}_{k-1}\bar{E}(k,n)+\bar{D}_{k}\bar{E}(k-1,n)
\end{equation}

and then shifting $n\rightarrow n+2k+1$ in \eqref{Dk2.1} we
obtain
\begin{equation}\label{Dk3}
\tilde{a}_{n+2}-\tilde{a}_{n-4}-\tilde{D}_k(\tilde{a}_n-\tilde{a}_{n-2})=\sum_{l=1}^{k}\frac{\widetilde{W}_{k,l}}{n+2l+1}:=\widetilde{Q}_k^{(1)}(n),
\end{equation}
where $\widetilde{W}_{k,l}$ is independent of $n$.
 
Now, for $k\neq \kappa\geq \tilde{k}+1$, the equations
\eqref{Dk101} and \eqref{Dk3} give, respectively,
\begin{equation}\label{Dk2}
(\tilde{D}_{\kappa}-\tilde{D}_{k})(a_n-a_{n-2})=Q_k^{(1)}(n)-Q_{\kappa}^{(1)}(n)
\end{equation}
and 
\begin{equation}\label{Dk4}
(\tilde{D}_{\kappa}-\tilde{D}_k)\left(\frac{n}{n+1}a_n-\frac{n-2}{n-1}a_{n-2}\right)=\widetilde{Q}_k^{(1)}(n)-\widetilde{Q}_{\kappa}^{(1)}(n).
\end{equation}

If $\tilde{D}_k\neq \tilde{D}_{\kappa}$ for some $k\neq \kappa\geq
\tilde{k}+1$, then by \eqref{Dk2} and \eqref{Dk4} we can
eliminate $a_{n-2}$ to get that $a_n$ is a rational fraction of
$n$. So, by Corollary~\ref{cor5}, we have $T_2=0$.

If $\tilde{D}_k= D$ for $k\geq \tilde{k}+1$, then
\eqref{Dk101} and \eqref{Dk3} become, respectively,
\begin{equation}\label{Dk6}
a_{n+2}-a_{n-4}-D(a_n-a_{n-2})=Q_k^{(1)}(n)
\end{equation}
and
\begin{equation}\label{Dk66}
\tilde{a}_{n+2}-\tilde{a}_{n-4}-D(\tilde{a}_n-\tilde{a}_{n-2})=\widetilde{Q}_k^{(1)}(n).
\end{equation}

The subtraction $Eq\eqref{Dk6}-Eq\eqref{Dk66}$ leads to
\begin{equation}
\frac{a_{n+2}}{n+3}-\frac{a_{n-4}}{n-3}-D\left(\frac{a_{n}}{n+1}-\frac{a_{n-2}}{n-1}\right)=Q_k^{(2)}(n).\label{Dk7}
\end{equation}
Then the combinations
$\left(Eq\eqref{Dk6}-(n+3)Eq\eqref{Dk7}\right)/2$ and
$(Eq\eqref{Dk6}-(n-3)Eq\eqref{Dk7})/2$ give, respectively,
\begin{equation}\label{Dk8}
\frac{3a_{n-4}}{n-3}-D\left(-\frac{a_n}{n+1}+\frac{2a_{n-2}}{n-1}\right)=Q_k^{(3)}(n)
\end{equation}
and
\begin{equation}\label{Dk9}
\frac{3a_{n+2}}{n+3}-D\left(\frac{2a_n}{n+1}-\frac{a_{n-2}}{n-1}\right)=Q_k^{(4)}(n).
\end{equation}

By shifting $n\rightarrow n+2$ in \eqref{Dk8} we obtain
\begin{equation}\label{Dk10}
\frac{3a_{n-2}}{n-1}-D\left(-\frac{a_{n+2}}{n+3}+\frac{2a_{n}}{n+1}\right)=Q_k^{(3)}(n+2).
\end{equation}

The elimination of $a_{n+2}$ and $a_{n-2}$ by the operations $D
Eq\eqref{Dk9}-3Eq\eqref{Dk10}$ and $3
Eq\eqref{Dk9}-DEq\eqref{Dk10}$, respectively, yields
\begin{equation}\label{Dk11}
\frac{6D-2D^2}{n+1}a_n+\frac{D^2-9}{n-1}a_{n-2}=Q_k^{(5)}(n)
\end{equation} 
and 
\begin{equation}\label{Dk12}
\frac{9-D^2}{n+3}a_{n+2}-\frac{6D-2D^2}{n+1}a_{n}=Q_k^{(6)}(n).
\end{equation} 
Finally, the shifting $n\rightarrow n-2$ in \eqref{Dk12} leads
to
\begin{equation}\label{Dk13}
\frac{9-D^2}{n+1}a_{n}-\frac{6D-2D^2}{n-1}a_{n-2}=Q_k^{(6)}(n-2)\end{equation}
and the operation $(6D-2D^2)Eq
\eqref{Dk11}+(D^2-9)Eq\eqref{Dk13}$ gives
\begin{equation}\label{Dk14}
[(6D-2D^2)^2+(D^2-9)^2]a_n=Q_k^{(7)}(n).
\end{equation}

According to manipulations made above, $Q_k^{(7)}(n)$ is a
partial fraction of $n$. So, if $D\neq 3$, $a_n$ is a partial
fraction of $n$ and then $T_2=0$.

Now, we explore the case $D=3$.
We have from \eqref{Dk8} and \eqref{Dk9}:%
\begin{equation}
Q_{k}^{(3)}\left( n\right) =\frac{1}{2}\left( Q_{k}^{(1)}\left(
n\right)
-\left( n+3\right) Q_{k}^{(2)}\left( n\right) \right)=
\frac{1}{2}\left((n+3)\widetilde{Q}_k^{(1)}(n)-(n+2)Q_{k}^{(1)}\left(
n\right)\right) \label{Eq5}
\end{equation}%
and%
\begin{equation}
Q_{k}^{(4)}\left( n\right) =\frac{1}{2}\left( Q_{k}^{(1)}\left(
n\right)
-\left( n-3\right) Q_{k}^{(2)}\left( n\right) \right)=
\frac{1}{2}\left((n-3)\widetilde{Q}_k^{(1)}(n)-(n-4)Q_{k}^{(1)}\left(
n\right)\right) . \label{Eq6}
\end{equation}%

Remark that $Q_{k}^{(j)}\left( n\right) ,1\leq j\leq 4,$ and
$\widetilde{Q}_k^{(1)}(n)$ are independent of $%
k$. Observe also that, according to the left-hand sides of
\eqref{Dk8} and \eqref{Dk9} for $D=3$, we have%
\begin{equation*}
Q_{k}^{(3)}\left( n+2\right) =Q_{k}^{(4)}\left( n\right) .
\end{equation*}%
From (\ref{Eq5}) and (\ref{Eq6}) we get  
\begin{equation}
(n+4)Q_{k}^{(1)}\left( n+2\right) -(n-4)Q_{k}^{(1)}\left(
n\right) =\left( n+5\right)
\widetilde{Q}_{k}^{(1)}\left( n+2\right) -\left( n-3\right)
\widetilde{Q}_{k}^{(1)}\left( n\right).
\label{Eq7}
\end{equation}%
In (\ref{Eq7}) two partial fractions are equal for natural
numbers and are so for real numbers. By using the expressions of
$Q_{k}^{(1)}\left( n\right)$ and $\widetilde{Q}_{k}^{(1)}\left(
n\right)$ (see \eqref{Dk101} and \eqref{Dk3}) we find that
\begin{eqnarray}\label{Eq7a}
(n+4)Q_{k}^{(1)}\left( n+2\right) -(n-4)Q_{k}^{(1)}\left(
n\right)
&=&\sum_{l=1}^{k}\frac{(n+4)W_{k,l}}{n+2-2k+2l}
-\sum_{l=1}^{k}\frac{(n-4)W_{k,l}}{n-2k+2l} \\
&=&\frac{2W_{k,k}}{n+2}-\frac{(2k-6)W_{k,1}}{n-2k+2}
+\sum_{l=2}^{k}\frac{(2k-2l+4)W_{k,l-1}-(2k-2l-4)W_{k,l}}{n-2k+2l}.\nonumber
\end{eqnarray}
and
\begin{eqnarray}\label{Eq7b}
(n+5)\widetilde{Q}_{k}^{(1)}\left( n+2\right)
-(n-3)\widetilde{Q}_{k}^{(1)}\left( n\right)
&=&\sum_{l=1}^{k}\frac{(n+5)\widetilde{W}_{k,l}}{n+2l+3}
-\sum_{l=1}^{k}\frac{(n-3)\widetilde{W}_{k,l}}{n+2l+1} \\
&=&\frac{6\widetilde{W}_{k,1}}{n+3}-\frac{(2k-2)\widetilde{W}_{k,k}}{n+2k+3}
+\sum_{l=2}^{k}\frac{(-2l+4)\widetilde{W}_{k,l-1}+(2l+4)\widetilde{W}_{k,l}}{n+2l+1}.\nonumber
\end{eqnarray}
Observe that the singularities of \eqref{Eq7a} are even
numbers, whereas the singularities of \eqref{Eq7b} are odd ones. So, we should have
$$W_{k,k}=W_{k,1}=\widetilde{W}_{k,1}=\widetilde{W}_{k,k}=0,$$
$$(2k-2l+4)W_{k,l-1}-(2k-2l-4)W_{k,l}=0$$
and
$$(-2l+4)\widetilde{W}_{k,l-1}+(2l+4)\widetilde{W}_{k,l}=0$$
\text{for} $2\leq l\leq k.$ 
Since $k\geq \tilde{k}+1\geq 4$ and by induction on $l$ all the $W_{k,l}$ and
$\widetilde{W}_{k,l}$ are null. Thus, \eqref{Dk6} reads
\begin{equation}
a_{n+2}-a_{n-4}-3\left( a_{n}-a_{n-2}\right) =0.  \label{Eq8}
\end{equation}%
The solution of (\ref{Eq8}) has the form
\begin{equation}\label{Eq8b}
a_{n}=\left( C_{1}+C_{2}n+C_{3}n^{2}\right)\left(-1\right)^{n}+C_{4}+C_{5}n+C_{6}n^{2}.
\end{equation}

Using \eqref{Eq8b} for $n$ even, the left-hand side of \eqref{gf12} is a partial fraction of $n$ with finite number of singularities. So, by the same arguments as in Corollary \ref{cor5} we obtain $T_{2}=0$.\\

{\bf\textit { Case 2:}} There exists $k_{0}\geq 3$ such that $D_{k}=0$ for $k\geq k_{0}$.

Suppose that $D_k= T_k^2-T_{k-1}T_{k+1}=0$ for all $k\geq k_0$.
First, notice that if there exists a $k_1\geq k_0$ such that $T_{k_1}=0$, then $T_{k_1-1}T_{k_1+1}=0$. So, $T_{k_1-1}=0$ or $T_{k_1+1}=0$ and by Corollary~\ref{cor3}, $T_2=0$. We have also $T_{k_0-1}\neq 0$, otherwise $T_{k_0}=0$ and  by Corollary~\ref{cor3}, $T_2=0$.

Now, for $T_k\neq 0$ $(k\geq k_0-1)$, we have 
\begin{equation}
\frac{T_{k+1}}{T_{k}}=\frac{T_{k}}{T_{k-1}}=\frac{T_{k_{0}}}{T_{k_{0}-1}}.
\label{Eq0.4}
\end{equation}%
This means that 
\begin{equation}\label{Eq1.4}
T_{k}=\left( \frac{T_{k_{0}}}{T_{k_{0}-1}}\right)^{k-k_{0}}T_{k_{0}}=ab^{k}
\end{equation}%
where $a=T_{k_{0}-1}^{k_{0}}/T_{k_{0}}^{k_{0}-1}$ and $b=T_{k_{0}}/T_{k_{0}-1}$.\\
The substitution of $T_{k}$ by $ab^{k}$ in \eqref{gf12} for $k\geq k_{0}$\ leads to the equation%
\begin{equation}
\frac{2}{T_1}b\left( a_{n}-\frac{n-2k-1}{n-2k}a_{n-2k-1}\right)
+\frac{n+2}{n}c_{n}-\frac{%
n-2k+1}{n-2k+2}c_{n-2k+1}=\frac{b^{-k}}{a}\sum_{l=1}^{k}\frac{%
T_{l}T_{k-l+1}}{n-2k+2l}=Q_{k}\left( n\right) .  \label{Eq4}
\end{equation}%
Let denote (\ref{Eq4}) by $\widetilde{E}\left( k,n\right)$ and make the
subtraction $\widetilde{E}\left( k+1,n+2\right) -\widetilde{E}\left( k,n\right)$ to get
\begin{equation}
\frac{2}{T_1}b\left( a_{n+2}-a_{n}\right)
+\frac{n+4}{n+2}c_{n+2}-\frac{n+2}{n}%
c_{n}=Q_{k+1}\left( n+2\right) -Q_{k}\left( n\right) .
\label{Eq11}
\end{equation}%
On the right hand side of  (\ref{Eq11})  we have, for $k\geq k_{0}$, the expression

\begin{eqnarray}
\widetilde{Q}_{k}\left( n\right) &:=&Q_{k+1}\left( n+2\right)
-Q_{k}\left(
n\right)
=\frac{b^{-k-1}}{a}\sum_{l=1}^{k+1}\frac{T_{l}T_{k-l+2}}{%
n-2k+2l}-\frac{b^{-k}}{a}\sum_{l=1}^{k}\frac{T_{l}T_{k-l+1}}{n-2k+2l}\nonumber\\
&=&\frac{b^{-k-1}}{a}\frac{T_{k+1}T_{1}}{n+2}+\frac{b^{-k-1}}{a}\frac{%
T_{k}\left( T_{2}-bT_{1}\right) }{n}+\frac{b^{-k-1}}{a}\sum_{l=1}^{k-1}\frac{T_{l}\left( T_{k-l+2}-bT_{k-l+1}\right)}{n-2k+2l}\nonumber\\
&=&\frac{T_{1}}{n+2}+\frac{T_{2}-bT_{1}}{bn}+\frac{b^{-k-1}}{a}\sum_{l=1}^{k-1}\frac{T_{l}\left(T_{k-l+2}-bT_{k-l+1}\right)}{n-2k+2l}
\end{eqnarray}
from which we deduce
\begin{eqnarray}
\widetilde{Q}_{k+1}\left( n\right)
&=&\frac{T_{1}}{n+2}+\frac{T_{2}-bT_{1}}{%
bn}+\frac{b^{-k-2}}{a}\sum_{l=1}^{k}\frac{T_{l}\left(T_{k-l+3}-bT_{k-l+2}\right) }{n-2k-2+2l}\nonumber\\
&=&\frac{T_{1}}{n+2}+\frac{T_{2}-bT_{1}}{bn}+\frac{b^{-k-2}}{a}\sum_{l=1}^{k-1}\frac{T_{l+1}\left(T_{k-l+2}-bT_{k-l+1}\right)}{n-2k+2l}.
\end{eqnarray}
Now  since the left hand side of equation  (\ref{Eq11}) is independent of $k$, it follows
\begin{equation}
\widetilde{Q}_{k+1}\left( n\right) -\widetilde{Q}_{k}\left(n\right) =\frac{b^{-k-2}}{a}\sum_{l=1}^{k-1}\frac{\left( T_{l+1}-bT_{l}\right)\left( T_{k-l+2}-bT_{k-l+1}\right) }{n-2k+2l}=0.
\end{equation}

As a result, for $1\leq l\leq k-1$ and $k\geq k_{0}$, we have
\begin{equation}
\left( T_{l+1}-bT_{l}\right) \left( T_{k-l+2}-bT_{k-l+1}\right)
=0.
\end{equation}
Let take $k=2\left( k_{0}-2\right) -1$ and $l=k_{0}-2$\ to get
$\left(
T_{k_{0}-1}-bT_{k_{0}-2}\right) ^{2}=0$ and then
$T_{k_{0}-1}=bT_{k_{0}-2},$ (or equivalently $D_{k_0-1}=0$).
Thus, the equations \eqref{Eq0.4} and \eqref{Eq1.4} are valid for $k=k_{0}-1$ and by induction we arrive at $T_{4}=bT_{3},$ (or equivalently $D_{4}=0$). For  $k=4$,  the right-hand side of \eqref{gf23} is null. So, $V_{4,2}=0$ and using $T_5=T_4^2/T_3$ (from $D_{4}=0$) we get $D_3=0$.

On the other side suppose that $T_2\neq 0$, then we can write
$$T_{k}=\left(\frac{T_3}{T_2}\right)^{k-2}T_{2}=ab^k,\;\;\text{for } k\geq 2,$$
where $b=T_3/T_2$ and $a=T_2^3/T_3^2$. Therefore, the equation \eqref{Eq4} reads
\begin{eqnarray}
\frac{2}{T_1}b\left(a_n-\frac{n-2k-1}{n-2k}a_{n-2k-1}\right)+\frac{n+2}{n}c_n-\frac{n-2k+1}{n-2k+2}c_{n-2k+1}=\nonumber\\
=\frac{T_1}{n-2k+2}+\frac{T_1}{n}+\sum_{l=2}^{k-1}\frac{ab}{n-2k+2l},\;\;\hbox{ for } k\geq 2 \hbox{ and } n\geq 2k+1.\label{gf1222}
\end{eqnarray}
When $n=2k+1$ and $n=2k+2$, the equation \eqref{gf1222} gives
\begin{equation}\label{gf12222}
\frac{2}{T_1}ba_{2k+1}+\frac{2k+3}{2k+1}c_{2k+1}=\frac{2}{3}c_2+\frac{T_1}{3}+\frac{T_1}{2k+1}+\sum_{l=2}^{k-1}\frac{ab}{2l+1}
\end{equation}
and
\begin{equation}
\frac{2}{T_1}ba_{2k+2}+\frac{k+2}{k+1}c_{2k+2}=\frac{1}{T_1}a_1b+\frac{3}{4}c_3+\frac{T_1}{4}+\frac{T_1}{2k+2}+\sum_{l=2}^{k-1}\frac{ab}{2l+2}\label{gf12223}
\end{equation}

respectively. Let take $n=2N+1$ in \eqref{gf1222} and use \eqref{gf12222} to obtain the expression
\begin{eqnarray}
&&\frac{2}{3}c_2+\frac{T_1}{3}+\frac{T_1}{2N+1}+\sum_{l=2}^{N-1}\frac{ab}{1+2l}-\frac{2(N-k)b}{2(N-k)+1}\frac{2}{T_1}a_{2(N-k)}-\frac{2(N-k)+2}{2(N-k)+3}c_{2(N-k)+2}\nonumber\\
&&=\frac{T_1}{2(N-k)+3}+\frac{T_1}{2N+1}+\sum_{l=2}^{k-1}\frac{ab}{2(N-k)+2l+1}.\nonumber
\end{eqnarray}

In this last equality let put   $N-k$  instead of $k$ to get
\begin{eqnarray}
-\frac{2bk}{2k+1}\frac{2}{T_1}a_{2k}-\frac{2k+2}{2k+3}c_{2k+2}&=&-\frac{2}{3}c_2-\frac{T_1}{3}-\sum_{l=2}^{N-1}\frac{ab}{1+2l}+\frac{T_1}{2k+3}+\sum_{l=2}^{N-k-1}\frac{ab}{2k+2l+1}\nonumber\\
&=&-\frac{2}{3}c_2-\frac{T_1}{3}+\frac{T_1}{2k+3}-\sum_{l=1}^{k}\frac{ab}{2l+3}.\label{gf12224}
\end{eqnarray}

After defining
$A_1=\frac{a_1}{T_1}+\frac{3}{4}\frac{c_3}{b}+\frac{T_1}{4b}$, $A_2=-\frac{2}{3}\frac{c_2}{b}-\frac{T_1}{3b}$ and $A_3=\frac{T_1}{b}$ and making the operation $$\frac{1}{k+2}\left(\frac{2k+2}{2k+3}\,Eq\eqref{gf12223}+\frac{k+2}{k+1}\,Eq\eqref{gf12224}\right)$$
we have
\begin{eqnarray}\label{an_tel}
&&\frac{2(k+1)}{(k+2)(2k+3)}\frac{2}{T_1}\,a_{2(k+1)}-\frac{2k}{(k+1)(2k+1)}\frac{2}{T_1}\,a_{2k}=\nonumber\\
&=&\frac{A_1(2k+2)+A_3}{(k+2)(2k+3)}
+\frac{A_2}{k+1}+\frac{A_3}{(k+1)(2k+3)}+\frac{2k+2}{(k+2)(2k+3)}\sum_{l=2}^{k-1}\frac{a}{2+2l}-\frac{1}{k+1}\sum_{l=2}^{k+1}\frac{a}{2l+1}\nonumber\\
 &=&-\,{\frac {{2 A_1}}{2\,k+3}}+{\frac {2\,{ A_2}-{ A_3}}{k+2}}+{\frac {{A_2}+{A_3}}{k+1}}+\left(-\frac{1}{2\,k+3}+\frac{1} {k+2}\right)\sum_{l=2}^{k-1}\frac{a}{l+1}-\frac{1}{k+1}\sum_{l=2}^{k+1}\frac{a}{2l+1}\nonumber\\
&=&\frac{ B_1}{k+2}+\frac{B_2}{k+\frac{3}{2}}+\frac { B_3}{k+1}+a \left( \frac{1}{ k+2} -\frac{1}{2}\, \frac{1}{k+\frac{3}{2}}\right) \Psi \left( k+1 \right) -\frac{1}{2}\,{\frac {a}{k+1}}\Psi\left( k+\frac{5}{2} \right),
\end{eqnarray}
where short notations $B_1=\left( -3/2+\gamma \right) a+2{A_1}-{ A_3}$, $B_2=\left( 3/4-\gamma/2\right) a-{ A_1}$, $B_3=\left( -\gamma/2-\ln \left( 2 \right) +4/3 \right) a+{ A_2}+{ A_3}$, ($\gamma$ is Euler's constant) are introduced as well as $\Psi(x)$ which stands for the Digamma function.

Taking $$U_k=\frac{2k}{(k+1)(2k+1)}\frac{2}{T_1}\,a_{2k}$$ and 
\begin{equation}
 G(k+1)
={\frac {{ B_1}}{k+2}}+{\frac {{ B_2}}{k+\frac{3}{2}}}+{\frac {{ 
B_3}}{k+1}}+a \left( \frac{1}{ k+2} -\frac{1}{2}\, \frac{1}{
k+\frac{3}{2}}\right) \Psi \left( k+1 \right) -\frac{1}{2}\,{\frac {a
}{k+1}}\Psi
 \left( k+\frac{5}{2} \right),
\end{equation}
then \eqref{an_tel} can be  written in compact form as
$$U_{k+1}-U_k=G(k+1).$$

The later recurrence is easily solved to give $$U_k=U_3+\sum^{k}_{j=4}G(j).$$

By using the formula $\Psi(j+1)=\Psi(j)+1/j$ 
and the relations \cite[Theorems 3.1 and 3.2]{milgram}

\begin{equation}
\sum_{l=0}^{k}\frac{\Psi(l+\alpha)}{l+\beta}+\sum_{l=0}^{k}\frac{\Psi(l+\beta+1)}{l+\alpha}=\Psi(k+\alpha+1)\Psi(k+\beta+1)-\Psi(\alpha)\Psi(\beta),
\end{equation}

\begin{equation}
\sum_{j=0}^{k}\frac{\Psi(j+\beta)}{j+\beta}=\frac{1}{2}\left[\Psi\, '(k+\beta+1)-\Psi\,'(\beta)+\Psi(k+\beta+1) ^2-\Psi(\beta)^2\right],
\end{equation}
we obtain

\begin{eqnarray}
U_k=\frac{2k}{(k+1)(2k+1)}\frac{2}{T_1}\,a_{2k}&=&\frac{a}{2}
\left( \Psi \left( k+2 \right) \right) ^{2}+ B_1 \Psi \left( k+2 \right)+ \frac{a}{2}\Psi\,'
\left( k+2 \right)+B_2 \Psi \left( k+\frac{3}{2} \right)
+\nonumber\\
&& \left( -\frac{a}{2} \Psi \left( k+\frac{3}{2} \right) + B_3
\right)\Psi \left( k+1 \right) + \frac{a}{ k+1}+\delta_2.\label{a2kasympt}
\end{eqnarray}
From \eqref{a2kasympt} we deduce the asymptotic behaviour of $a_{2k}$ as
$k\to\infty$:
\begin{equation}\label{a2kk}
\frac{2}{T_1}\,a_{2k}=\left(\delta _{1}\left(k +\frac{3}{2}+
\frac{1}{2k}\right)+\frac{3a}{4} + {\frac {5a}{16 k}}+
\frac{a}{32 k^2}-{\frac {3a}{128 k^3}+...}\right)\ln(k)+\delta
_{2}k+\delta _{3}+\frac{\delta _{4} }{k}+\frac{\delta _{5}
}{k^{2}}+\frac{\delta_6}{k^3}+\cdots
\end{equation}
where  coefficients $\delta _{i}$ are defined by (higher terms are omitted)
\begin{eqnarray}\nonumber
&&\delta _{1}=B_1+B_2+B_3,\\\nonumber
&&\delta_2 =
\left(\gamma- {\frac {25}{12}} \right){ B_1} + \left(\gamma+2\ln ( 2 )- {\frac {352}{105}} \right){ B_2} +
 \left(\gamma- {\frac {11}{6}} \right){ B_3} \\\nonumber
&&\qquad+\left(\left(\ln ( 2 ) -{\frac {107}{210}} \right)
\gamma+{
\frac {3439}{2520}}-{\frac {11}{6}}\,\ln  \left( 2
 \right)-\frac{{\pi }^{2}}{12}\right)\,a+\frac {3}{7\,T_1}\,
{a_6},
\\\nonumber
&&\delta_3=\frac{3}{2}\,B_1+{ B_2}+\frac{
1}{2}\,B_3+\frac{3}{2}\,a+\frac{3}{2}\,\delta_2,
\\\nonumber
&&\delta_4=\frac{7}{6}\,{ B_1}+{\frac {25}{24}}\,{
 B_2}+\frac{2}{3}\,{ B_3}+{\frac {11}{8}}\,a+\frac{1}{2}\,\delta_2,
\\\nonumber
&&\delta_5=\frac{1}{8}\,{ B_1}+\frac{1}{16}\,{ B_2}+\frac{1}{8}\,{
B_3}+{\frac {5}{96}}\,a,\\\nonumber
&&\delta_6=-\frac{1}{30}\,{ B_1}+{\frac{13}{960}
}\,{ B_2}-\frac{1}{30}\,{ B_3}-{\frac{1}{96}}\,a,\\\nonumber
&&\quad\vdots \nonumber
\end{eqnarray}

At this step we should remark that $\lim_{k\to \infty}a_{2k}=\infty$ for all
$\delta_i$, $i=1,2,3,...,$ since $a\neq 0$.

Recall that
$c_n=T_1\left(na_n/a_{n-1}-(n-1)\right)/2$, then the
equation \eqref{gf12222} can be written as
\begin{equation}\label{phi}
a_{2k+1}\left(b\frac{2}{T_1}+\frac{T_1}{2}\frac{2k+3}{a_{2k}}\right)=\phi(k),
\end{equation}
where $$\phi \left( k\right)
=\frac{T_1}{2}\frac{2k(2k+3)}{2k+1}-bA_2+\frac{bA_3}{2k+1}+\sum_{l=2}^{k-1}\frac{ab}{2l+1}.$$

$\bullet$ If we suppose  $\lim_{k\to\infty} \frac{a_{2k}}{2k}=\infty ,$ then from \eqref{phi} we deduce on one side%
\begin{equation}
\lim_{k\to\infty} \frac{a_{2k+1}}{2k+1}=\lim_{k\to\infty}
\frac{\frac{\phi \left( k\right) }{2k+1}}{%
\frac{2b}{T_{1}}+\frac{T_{1}}{2}\frac{2k+3}{a_{2k}}}=\frac{T_{1}^{2}}{4b}.
\label{e1}
\end{equation}
On the other side, for $n=2k+3$, \eqref{gf11} reads
\begin{equation}
\frac{4T_2}{T_1^3}\left(
1-\frac{2k}{2k+1}\frac{a_{2k}}{a_{2k+3}}\right)
=\frac{2k+4}{a_{2k+3}%
}-2\frac{2k+3}{a_{2k+2}}+\frac{2k+2}{a_{2k+1}}.  \label{e2}
\end{equation}
Under the assumption $T_2\neq 0$, \eqref{e2} admits the limit $\infty =8b/T_{1}^{2}$, as $k\to\infty$,  which exhibit a contradiction.

$\bullet$ Now if $\lim_{k\to\infty} \frac{a_{2k}}{2k}=\eta_{1}\neq 0,$ then from \eqref{phi} we have
\begin{equation}
\lim_{k\to\infty} \frac{a_{2k+1}}{2k+1}=\lim_{k\to\infty}
\frac{\frac{\phi \left( k\right) }{2k+1}}{\frac{2b}{T_{1}}+\frac{T_{1}}{2}\frac{2k+3}{a_{2k}}}=\frac{\frac{T_{1}}{2}}{\frac{2b}{T_{1}}+\frac{T_{1}}{2\eta_{1}}}:=\eta _{2}.
\label{e3}
\end{equation}
Equation \eqref{gf11} becomes, for $n=2k+2$, 
\begin{equation}
\frac{4T_2}{T_1^3}\left(
1-\frac{2k-1}{2k}\frac{a_{2k-1}}{a_{2k+2}}\right) =\frac{2k+3}{a_{2k+2}}-2\frac{2k+2}{a_{2k+1}}+\frac{2k+1}{a_{2k}}. \label{e4}\end{equation}
And if we assume that $T_2\neq 0$ and $\eta _{2}=\infty ,$ then the limit process $k\to\infty$ in (\ref{e4}) left us with  the contradiction $\infty =2/\eta _{1}.$
But if $\eta _{2}\neq \infty ,$ then by taking the limit in (\ref{e2}) and (\ref{e4}) we obtain, respectively,
\[
\frac{4T_2}{T_1^3}\left( 1-\frac{\eta _{1}}{\eta_{2}}\right) =\frac{2}{\eta _{2}}-\frac{2}{\eta _{1}}
\]
and
\[
\frac{4T_2}{T_1^3}\left( 1-\frac{\eta _{2}}{\eta_{1}}\right) =\frac{2}{\eta _{1}}-\frac{2}{\eta _{2}}.
\]
Adding the two later we get $2-\eta _{1}/\eta _{2}-\eta_{2}/\eta _{1}=0$ and therefore $\eta _{1}=\eta _{2}$. According to (\ref{e3}) $\eta _{1}=0$ which is in contradiction with the initial hypothesis $\eta _{1}\neq 0.$

$\bullet$ Finally if $\lim _{k\to\infty}\frac{a_{2k}}{2k}=0,$ then from
 (\ref{e3}) and \eqref{phi}  we  have respectively  $\lim_{k\to\infty} \frac{a_{2k+1}}{2k+1}=0$ and
\[
\lim_{k\to\infty} \frac{a_{2k+1}}{a_{2k}}=\lim_{k\to\infty}
\frac{\frac{\phi \left( k\right) }{2k}}{\frac{2b}{T_{1}}\frac{a_{2k}}{2k}+\frac{T_{1}}{2}\frac{2k+3}{2k}}=1.
\]

Similarly, from \eqref{gf12223} we obtain $\lim_{k\to\infty}
\frac{a_{2k+2}}{a_{2k+1}}=1.$ Now, since  $\lim_{n\to\infty}
\frac{a_{n}}{n}=0$ and $\lim_{n\to\infty}
\frac{a_{n+1}}{a_{n}}=1$, the left-hand side of \eqref{e4} tends to $0$ as $k\to\infty$. In the other hand, according to \eqref{a2kk},
$\lim _{k\to\infty}\frac{a_{2k}}{2k}=0$ implies that
$\delta_1=\delta_2=0$ and
\begin{eqnarray}\label{a2kkk}
a_{2k}&=&\frac{T_1}{2}\frac{(k+1)(2k+1)}{2k}\left(\frac{a}{2}\left( \Psi \left( k+2 \right) \right) ^{2}+ B_1 \Psi \left( k+2 \right)+ \frac{a}{2}\Psi\,'
\left( k+2 \right)+B_2 \Psi \left( k+\frac{3}{2} \right)
+\right.\nonumber\\
&& \left.\left( -\frac{a}{2} \Psi \left( k+\frac{3}{2}\right) -
B_1-B_2 \right)\Psi \left( k+1 \right) + \frac{a}{k+1}\right),\nonumber\\
&=&\frac{T_1}{2}\left(\left(\frac{3a}{4} + {\frac {5a}{16 k}}+
\frac{a}{32 k^2}-{\frac {3a}{128 k^3}+...}\right)\ln(k)+\delta
_{3}+\frac{\delta _{4} }{k}+\frac{\delta _{5}
}{k^{2}}+\frac{\delta_6}{k^3}+\cdots \right).
\end{eqnarray}
From \eqref{phi} we have
\begin{equation}\label{a2k1}
a_{2k+1}=\frac{\phi(k)}{b\frac{2}{T_1}+\frac{T_1}{2}\frac{2k+3}{a_{2k}}},
\end{equation}
which gives an explicit formula for $a_{2k+1}$. Using \eqref{a2k1}, the right-hand side of \eqref{e4} reads

\begin{equation}
-8\,\frac{b}{T_1}{\frac { k+1 }{\phi \left( k \right) }}+{\frac
{2\,k+3}{a_{2 k+2} }}+ \left( -2\,{ T_1}{\frac {\left( 2k+3 \right) \, \left( k+1 \right) }{\phi \left( k\right) }}+2\,k+1 \right)  \frac{1}{ a_{2k}}.
\end{equation}

By virtue of \eqref{a2kkk}, the limit of the both sides of
\eqref{e4}, as $k\to\infty$, gives $-\frac{8b}{3T_1^2}=0$. So, $b=0$ and $T_k=0$ for $k\geq 3$. Therefore, by corollary \ref{cor3} we have $T_{2}=0$ which contradicts $T_{2}\neq 0$.\\

{\bf\textit{Case 3:}} For every $k_{0}\geq 3$, there exists $k\geq k_{0}$ such that
$D_{k}=0$ or $D_{k}\neq 0$.

To exclude {\it Case 1} and {\it Case 2},
there exists a mixed case with infinitely many $k$ and $\kappa $ such that: $D_{k}=0$
and $D_{\kappa }\neq 0.$ Now, it suffices to take $k_{1}$ and
$k_{2}$, $k_{1}\neq k_{2}$, with $D_{k_{1}}=0$, $D_{k_{1}+1}\neq 0$,
$D_{k_{2}}=0$ and
$D_{k_{2}+1}\neq 0$ to get two equations similar to \eqref{gf16} and \eqref{gf17}.
Consequently, a reasoning analogous to that of Corollary
\ref{cor6} completes the proof.
\begin{flushright}
$\blacksquare$
\end{flushright}

\section{Concluding remarks}\label{sec4}

We have shown that the only polynomial sets (besides the monomial set) $\{P_n\}$ generated by $F(xt-R(t))=\sum_{n\geq 0}\alpha_nP_n(x)t^n$ and satisfying the three-term recursion $xP_n(x)=P_{n+1}(x)+\beta_nP_n(x)+\omega_nP_{n-1}(x)$, are the rescaled ultraspherical, Hermite and Chebychev polynomials of the first kind. In \cite{Bencheikh}, the authors generalized the results obtained in \cite{alsalam} and \cite{Bachhaus} in the context of $d$-orthogonality by considering the polynomials (which fulfils a $(d+1)$-order difference equation) generated by $F((d+1)xt-t^{d+1})$, where $d$ is a positive integer. Recently in \cite{Varma}, the author characterized the Shefer $d$-orthogonal polynomials. These polynomials have the generating function $A(t)\exp(xH(t))$ which has the alternative form $\exp(xH(t)+\ln(A(t)))=F(xU(t)-R(t))$.
So, a natural extension is to look at polynomial sets generated by $F(xU(t)-R(t))$ and satisfying the $(d+1)$-order recursion 
\begin{equation}
xP_n(x)=P_{n+1}(x)+\sum_{l=0}^{d}\gamma_{n}^{l}P_{n-l}(x),\label{dr}
\end{equation}
where $\{\gamma_{n}^{l}\}$, $0\leq l\leq d$, are complex sequences.\\
Currently, we are attempting to generalize the results given here by investigating polynomial sets satisfying the recursion \eqref{dr} and generated by $F(xt-R(t))$. This also provides generalizations of the results given in \cite{Bencheikh} and \cite{Varma}.\\

{\bf Acknowledgements:} One of us (M. B. Z.) would like to thank Prof. Dominique Manchon for precious help and useful discussions and for his  high hospitality at "Laboratoire de
Math\'ematiques, CNRS-UMR 6620 " (Clermont-Ferrand).

\end{document}